\documentclass[a4paper,11pt]{amsart}
\usepackage{hyperref,fancyhdr,mathrsfs,amsmath,amscd,amsthm,amsfonts,latexsym,amssymb,stmaryrd}
\usepackage[all]{xy}

\DeclareMathOperator{\rank}{rank}

\DeclareMathOperator{\dif}{d}

\newcommand{\Cal}{\mathcal{C}}
\renewcommand{\H}{\mathscr{H}}

\newcommand{\Fa}{\mathcal{F}}
\newcommand{\J}{\mathcal{J}}

\newcommand{\tD}{\mathscr{D}}
\newcommand{\tB}{\mathscr{B}}
\newcommand{\ol}{\mathcal{O}}

\def \a{\alpha}

\def \G{\Gamma}
\def \l{\lambda}

\def \O{\Omega}
\def \phi{\varphi}
\def \Phi{\varPhi}
\def \p{\pi}
\def \r{\rho}
\def \s{\sigma}
\def \t{\tau}

\def \R{\mathbb{R}}
\def \Hq{\mathbb{H}\,}
\def \C{\mathbb{C}\,}

\def\widecheckg{g^{\hspace*{-2.5pt}\vbox to 5pt{\hbox to
0pt{\LARGE$\check{}$}}}\hspace*{2pt}}

\def\widecheckl{\lambda^{\hspace*{-3.5pt}\vbox to 8pt{\hbox to
0pt{\LARGE$\check{}$}}}\hspace*{2pt}}

\begin{document}

\title{Twistor Theory for CR quaternionic manifolds and related structures}
\author{S.~Marchiafava, L.~Ornea, R.~Pantilie}
\thanks{L.O.\ and R.P.\ acknowledge that this work was partially supported by CNCSIS (Rom\^ania), through the Project no.\ 529/06.01.2009,
PN II - IDEI code~1193/2008;
also, R.P.\ acknowledges that this work was partially supported by the Visiting Professors Programme of GNSAGA-INDAM of C.N.R. (Italy),
and by the ``Agreement for direct cultural and scientific cooperation'' between Rome and Bucharest.\\
\indent
S.M.\ acknowledges that this work was done under the program of GNSAGA-INDAM of C.N.R. and PRIN07 ``Geometria Riemanniana e strutture
differenziabili'' of MIUR (Italy).}
\email{\href{mailto:marchiaf@mat.uniroma1.it}{marchiaf@mat.uniroma1.it},
       \href{mailto:liviu.ornea@imar.ro}{liviu.ornea@imar.ro},
       \href{mailto:radu.pantilie@imar.ro}{radu.pantilie@imar.ro}}
\address{S.~Marchiafava, Dipartimento di Matematica, Istituto ``Guido~Castelnuovo'',
Universit\`a degli Studi di Roma ``La Sapienza'', Piazzale Aldo~Moro, 2 - I 00185 Roma - Italia}
\address{L.~Ornea, Universitatea din Bucure\c sti, Facultatea de Matematic\u a, Str.\ Academiei nr.\ 14,
70109, Bucure\c sti, Rom\^ania, \emph{also,}
 Institutul de Matematic\u a ``Simion~Stoilow'' al Academiei Rom\^ane,
C.P. 1-764, 014700, Bucure\c sti, Rom\^ania}
\address{R.~Pantilie, Institutul de Matematic\u a ``Simion~Stoilow'' al Academiei Rom\^ane,
C.P. 1-764, 014700, Bucure\c sti, Rom\^ania}

\newtheorem{thm}{Theorem}[section]
\newtheorem{lem}[thm]{Lemma}
\newtheorem{cor}[thm]{Corollary}
\newtheorem{prop}[thm]{Proposition}

\theoremstyle{definition}

\newtheorem{defn}[thm]{Definition}
\newtheorem{rem}[thm]{Remark}
\newtheorem{exm}[thm]{Example}

\numberwithin{equation}{section}

\maketitle
\thispagestyle{empty}
\vspace{-10mm}
\section*{Abstract}
\begin{quote}
{\footnotesize
In a general and non metrical framework, we introduce the class of CR quaternionic manifolds containing the class of quaternionic manifolds,
whilst in dimension three it particularizes to, essentially, give the conformal manifolds.
We show that these manifolds have a rich natural Twistor Theory and, along the way, we obtain a heaven space construction
for quaternionic ma\-ni\-folds.}
\end{quote}

\section*{Introduction}

\indent
In this paper, we introduce the class of \emph{CR quaternionic manifolds} which contains the class of
quaternionic manifolds whilst in dimension three it, essentially, reduces to the class of conformal manifolds endowed with
the twistorial structure of~\cite{LeB-CR_twistors}\,.
Particular classes of such manifolds were already considered, see for example \cite{Bi-qK_heaven}\,,
\cite{AleKam-Annali08}\,, \cite{BejFar}\,, under certain dimensional assumptions and/or in a metrical framework.\\
\indent
To define the corresponding notion of almost CR quaternionic structure we introduce
and study the \emph{CR quaternionic vector spaces}.
The idea of the definition of these structures is based on the fact that
a linear quaternionic structure on a real vector space $E$ has an
associated $2$-sphere $Z$ of `admissible' linear complex structures on $E$
(the left multiplications by unit imaginary quaternions). Then a linear CR quaternionic
structure on a vector space $U$ is given by an injective linear map $\iota$ from $U$
into a quaternionic vector space $E$ such that, for any admissible $J\in Z$,
we have ${\rm im}\,\iota+J({\rm im}\,\iota)=E$\,; that is, $(E,J,\iota)$ corresponds to a linear CR structure on $U$
(see Section \ref{section:lin_cr_co-cr}\,).
Further, by duality we obtain the notion of \emph{linear co-CR quaternionic structure} (see Section \ref{section:lin_f-q}\,).\\
\indent
One of the main ingredients in the study of these vector spaces is a faithful covariant
functor from the category determined by them to the category of holomorphic vector bundles over the sphere
(Theorem \ref{thm:cr_co-cr_q_vector_spaces}\,). This leads to the classification of (co-)CR quaternionic vector spaces
(Corollary \ref{cor:classification_cr_q}\,).\\
\indent
To define the adequate notion of integrability for the almost CR quaternionic structures recall that an almost quaternionic structure
is integrable (in the sense of \cite{Sal-dg_qm}\,) if and only if there exists a compatible connection such that the associated
almost complex structure on the space of admissible linear complex structures be integrable (see \cite[Remark 2.10(2)\,]{IMOP}\,).
The straight generalization of this fact leads to the introduction of a suitable almost CR structure (see the paragraph before the
Definition \ref{defn:cr_q}\,) on the space of admissible linear complex structures of the quaternionic bundle defining an almost
CR quaternionic structure.\\
\indent
This construction is motivated, also, by Theorem \ref{thm:cr_q}\,.
Furthermore, similarly to the three-dimensional case, the twistor space of a CR quaternionic manifold $M$ is a CR manifold $Z$.
We prove (Corollary \ref{cor:realiz_cr_q_1}\,) that any real-analytic CR quaternionic structure on a manifold $M$
is induced by an embedding of $M$ in a (germ unique) quaternionic manifold $N$
such that $TN|_M$ is generated, as a quaternionic vector bundle, by $TM$.\\
\indent
We mention that the notions introduced in this paper are discussed and compared through several examples. Also, some
embedding and decomposition results, reducing the study to special examples of CR quaternionic structures,
are given (Corollary \ref{cor:classification_cr_q}\,, Theorem \ref{thm:realiz_cr_q}\,, and Corollary \ref{cor:realiz_cr_q_1}\,).\\
\indent
It is well known (see, for example, \cite{PanWoo-sd} and the references therein) that the study of the twistorial properties
of the three-dimensional conformal manifolds, endowed with a conformal connection, significantly contributed to a better understanding of
the anti-self-dual (Einstein) manifolds. One of the aims of this paper is to give a first indication that the study of
CR quaternionic manifolds will lead to a better understanding of quaternionic(-K\"ahler) manifolds.\\
\indent
The \emph{co-CR quaternionic manifolds}, which appear as the natural generalizations of both the quaternionic manifolds and
the three-dimensional Einstein--Weyl spaces, will be studied elsewhere \cite{fq_2}\,.

\section{Linear CR and co-CR structures} \label{section:lin_cr_co-cr}

\indent
Unless otherwise stated, all the vector spaces are assumed real and finite dimensional
and all the linear maps are assumed real.\\
\indent
A \emph{linear complex structure} on a vector space $U$ is a linear map $J:U\to U$ such that
$J^2=-{\rm Id}_U$\,. Then $U^{\C}$ is the direct sum of the $\pm{\rm i}$ eigenspaces of $J$.
Hence, any linear complex structure on $U$ is given by a complex vector subspace
$C\subseteq U^{\C}$ such that $U^{\C}=C\oplus\overline{C}$. More generally, we have the following,
mutually dual, definitions.

\begin{defn} \label{defn:lin_(co-)cr_str}
Let be $U$  a vector space and $C\subseteq U^{\C}$ a complex vector subspace.\\
\indent
1) $C$ is a \emph{linear CR structure} on $U$ if $C\cap\overline{C}=\{0\}$.\\
\indent
2) $C$ is a \emph{linear co-CR structure} on $U$ if $C+\overline{C}=U^{\C}$.
\end{defn}

\indent
A complex vector subspace $C\subseteq U^{\C}$ is a
linear CR structure on $U$ if and only if its annihilator $\bigl\{\,\a\in\bigl(U^{\C}\bigr)^{\!*}\,|\,\a|_C=0\,\bigr\}$
is a linear co-CR structure on $U^{\!*}$ .\\
\indent
Let $U$ and $U'$ be vector spaces endowed with linear (co-)CR structures $C$ and $C'$, respectively.
A \emph{(co-)CR linear map} is a linear map $t:(U,C)\to(U',C')$ such that $t(C)\subseteq C'$; if, further, $t$ is injective then
we say that $(U,C)$ is a (co-)CR vector subspace of $(U',C')$\,.

\begin{prop} \label{prop:cr_triple}
For any CR vector space $(U,C)$ we have the following facts:\\
\indent
{\rm (i)} The canonical map $\iota:U\to U^{\C}\!/C$ is injective.\\
\indent
{\rm (ii)} $C=\iota^{-1}\bigl({\rm ker}(J+{\rm i})\bigr)$\,, where $J$ is the linear complex structure of $E=U^{\C}\!/C$.\\
\indent
{\rm (iii)} ${\rm im}\,\iota+J({\rm im}\,\iota)=E$.\\
\indent
{\rm (iv)} $(E,\iota)$ is unique, up to complex linear isomorphisms, with the properties {\rm (i)}\,, {\rm (ii)} and {\rm (iii)}\,.
\end{prop}
\begin{proof}
The first assertion is equivalent to $C\cap\overline{C}=\{0\}$\,.\\
\indent
To prove (ii)\,, let $\iota^{1,0}$ be the composition of the complexification of $\iota$ followed by the projection from $E^{\C}$
onto the ${\rm i}$ eigenspace of $J$. Then $\iota^{1,0}(u\otimes\l)=\l\,\iota(u)=u\otimes\l+C$\,, for any $u\in U$ and $\l\in\C$.
Thus, $\iota^{-1}\bigl({\rm ker}(J+{\rm i})\bigr)={\rm ker}\bigl(\iota^{1,0}\bigr)=C$.\\
\indent
Assertion (iii) is an immediate consequence of the fact that the complex vector space $U^{\C}$ is generated by $U$.\\
\indent
Let $(E',J')$ be a complex vector space and let $\iota':U\to E'$ be an injective linear map such that
(1) $C=\iota^{-1}\bigl({\rm ker}(J'+{\rm i})\bigr)$ and (2) ${\rm im}\,\iota'+J'({\rm im}\,\iota')=E'$.
Then the complex extension of $\iota'$ to $U^{\C}$ is, by (2)\,, surjective and, by (1)\,, its kernel is $C$.
This shows that (iv) holds.
\end{proof}

\indent
The following two results are immediate consequences of Proposition \ref{prop:cr_triple}\,.

\begin{cor} \label{cor:cr_triple}
Any linear CR structure on a vector space $U$ corresponds to a pair $(E,\iota)$\,, unique up to complex linear isomorphisms,
where $E$ is a complex vector space and $\iota:U\to E$ is an injective linear map
such that ${\rm im}\,\iota+J({\rm im}\,\iota)=E$, with $J$ the linear complex structure of $E$.
\end{cor}

\begin{cor} \label{cor:lin_cr_map_triple}
Let $C$ and $C'$ be linear CR structures on $U$ and $U'$ corresponding to the pairs $(E,\iota)$ and $(E',\iota')$\,, respectively.\\
\indent
Let $t:U\to U'$ be a map. Then the following assertions are equivalent:\\
\indent
\quad{\rm (i)} $t$ is a CR linear map.\\
\indent
\quad{\rm (ii)} There exists a complex linear map $\widetilde{t}:E\to E'$ with the property
$\iota'\circ t=\widetilde{t}\circ\iota$\,.\\
\indent
Furthermore, if\/ {\rm (i)} holds then there exists a unique $\widetilde{t}$ satisfying {\rm (ii)}\,.
\end{cor}

\indent
The duals of Proposition \ref{prop:cr_triple} and of Corollaries \ref{cor:cr_triple} and \ref{cor:lin_cr_map_triple}
can be easily formulated.

\section{CR quaternionic and co-CR quaternionic vector spaces} \label{section:lin_f-q}

\indent
Let $\Hq$ be the associative algebra of quaternions. The automorphism group of $\Hq$ is
${\rm SO}(3)$ acting trivially on $\R$ and canonically on ${\rm Im}\Hq\,(=\R^3)$\,.\\
\indent
A \emph{linear quaternionic structure} on a vector space $E$ is an equivalence
class of morphisms of associative algebras from $\Hq$ to ${\rm End}(V)$ where two
such morphisms $\s$ and $\t$ are equivalent if there exists $a\in{\rm SO}(3)$
such that $\t=\s\circ a$\,. A \emph{quaternionic vector space} is a vector
space endowed with a linear quaternionic structure \cite{AleMar-Annali96} (see \cite{IMOP}\,).\\
\indent
Let $E$ be a quaternionic vector space and let $\s$ be a representative of its linear quaternionic structure.
Obviously, the unit sphere $Z$ of $\s({\rm Im}\Hq)$ depends only of the linear quaternionic structure of $E$\,;
any $J\in Z$ is an \emph{admissible linear complex structure} on $E$.\\
\indent
Also, $\s^*:\Hq\to{\rm End}(E^*)$ such that $\s^*(q)$ is the transpose of $\s(\overline{q})$, $(q\in\Hq)$\,,
defines a linear quaternionic structure on $E^*$ which depends only of the linear quaternionic structure of $E$.

\begin{defn} \label{defn:lin_q_str}
A \emph{linear CR quaternionic structure} on a vector space $U$ is a pair $(E,\iota)$ where
$E$ is a quaternionic vector space and $\iota:U\to E$ is an injective linear map
such that ${\rm im}\,\iota+J({\rm im}\,\iota)=E$, for any $J\in Z$.\\
\indent
A \emph{linear co-CR quaternionic structure} on $U$ is a pair $(E,\r)$ with $\r:E\to U$
a (surjective) linear map such that $(E^*,\r^*)$ is a linear CR quaternionic structure on $U^*$.\\
\indent
A \emph{(co-)CR quaternionic vector space} is a vector space endowed with a linear (co-)CR quaternionic structure.
\end{defn}

\begin{exm} \label{exm:U_k}
1) Let $q_1\,,\ldots,q_{k+1}\in S^2$, $(k\geq1)$\,, be such that $q_i\neq\pm q_j$\,, if $i\neq j$\,.
For $j=1,\ldots,k$ let $e_j=(\underbrace{0,\ldots,0}_{j-1},q_j,q_{j+1},\underbrace{0,\ldots,0}_{k-j})$\,.
Denote $V_0=\R$ and, for $k\geq1$\,, let $V_k=\R^{k+1}+\R e_1+\ldots+\R e_k$\,. If we denote $U_k=V_k^{\perp}$
then $(U_k,\Hq^{\!k+1})$ is a CR quaternionic vector space; note that, $\dim U_k=2k+3$\,, $U_0={\rm Im}\Hq$.\\
\indent
2) Let $V'_0=\{0\}$ and, for $k\geq1$\,, let $V'_k$ be the vector subspace of $\Hq^{\!2k+1}$ formed of all vectors of the form
$(z_1\,,\overline{z_1}+z_2\,{\rm j}\,, z_3-\overline{z_2}\,{\rm j}\,,\ldots, \overline{z_{2k-1}}+z_{2k}\,{\rm j}\,,-\overline{z_{2k}}\,{\rm j})\,,$
where $z_1\,,\ldots,z_{2k}$ are complex numbers. If we denote $U'_k={V'_k}^{\perp}$ then $(U'_k,\Hq^{2k+1})$ is a CR quaternionic vector space.
Note that, $\dim U'_k=4k+4$\,, and $U'_0=\Hq$.\\
\indent
3) Let $E=\Hq^{\!k}$ and let $U=({\rm Im}\Hq\!)^l\times\Hq^{\!k-l}$, for some $k\geq l$\,.
Then $(U,E)$ is a CR quaternionic vector space.
\end{exm}

\indent
Recall (see \cite{IMOP}\,) that a linear map $t:E\to E'$ between quaternionic vector spaces is \emph{quaternionic},
with respect to some map $T:Z\to Z'$ between the spaces of admissible linear complex structures on $E$ and $E'$,
if $t\circ J=T(J)\circ t$, for any $J\in Z$. Then, if $t\neq0$\,, we have that $T$ is unique and an orientation
preserving isometry.\\
\indent
A map $t:(U,E,\iota)\to(U',E',\iota')$ between CR quaternionic vector spaces is \emph{CR quaternionic linear}
(with respect to some map $T:Z\to Z'$\,) if there exists a map $\widetilde{t}:E\to E'$ which is quaternionic linear
(with respect to $T$) such that $\iota'\circ t=\widetilde{t}\circ\iota$.\\
\indent
By duality, we obtain the notion of \emph{co-CR quaternionic linear map}.\\
\indent
Note that, if $(U,E,\iota)$ is a CR quaternionic vector space then the inclusion $\iota:U\to E$ is CR quaternionic linear.
Dually, if $(U,E,\r)$ is a co-CR quaternionic vector space then the projection $\r:E\to U$ is co-CR quaternionic linear.

\newpage

\section{(Co-)CR quaternionic vector spaces\\
and holomorphic bundles over the sphere} \label{section:cr_q_co-cr_q_lin_map}

\indent
Let $(U,E,\iota)$ be a CR quaternionic vector space and let $Z(=S^2)$ be the space of admissible linear
complex structures of $E$. For $J\in Z$, we denote $U^J=\iota^{-1}\bigl(E^J\bigr)$\,, where $E^J$ is the eigenspace of $J$ corresponding
to $-{\rm i}$\,. Dually, $U^J=\r\bigl(E^J\bigr)$ for a co-CR quaternionic vector space $(U,E,\r)$\,.

\begin{prop} \label{prop:cr_q_lin_map}
Let $(U,E,\iota)$ and $(U',E',\iota')$ be CR quaternionic vector spaces. Let $t:U\to U'$ be a nonzero linear map
and let $T:Z\to Z'$ be a map.\\
\indent
Then the following assertions are equivalent:\\
\indent
\quad{\rm (i)} $t$ is CR quaternionic, with respect to $T$.\\
\indent
\quad{\rm (ii)} $T$ is a holomorphic diffeomorphism and, for any $J\in Z$\,, we have
\begin{equation} \label{e:cr_q_lin_map}
t\bigl(U^J\bigr)\subseteq(U')^{T(J)}\;.
\end{equation}
\indent
Furthermore, if assertion {\rm (i)} or {\rm (ii)} holds then there exists a unique linear map
$\widetilde{t}:E\to E'$ which is quaternionic, with respect to $T$, and such that
$\iota'\circ t=\widetilde{t}\circ\iota$\,. \\
\indent
The dual statement holds true, for nonzero linear maps between co-CR quaternionic vector spaces.
\end{prop}
\begin{proof}
It is obvious that if (i) holds then \eqref{e:cr_q_lin_map} holds, for any $J\in Z$.
If, further, $t\neq0$, the fact that $T$ is a holomorphic diffeomorphism is an immediate consequence
of \cite[Proposition 1.5]{IMOP}\,. This proves (i)$\Longrightarrow$(ii)\,.\\
\indent
To prove the converse, let $\mathcal{E}=Z\times E\to Z$ be the complex vector bundle
whose fibre over each $J\in Z$ is $(E,J)$\,. Alternatively, $\mathcal{E}$ is the quotient
of the trivial holomorphic vector bundle $Z\times E^{\C}$ through the holomorphic vector bundle $E^{0,1}\to Z$,
whose fibre over each $J\in Z$ is $E^J$. It follows that $\mathcal{E}$ is a holomorphic vector bundle,
over $Z=\C\!P^1$, isomorphic to $2k\ol(1)$\,, where $\dim E=4k$ and $\ol(1)$
is the dual of the tautological line bundle over $\C\!P^1$. Moreover, from the long exact sequence of
cohomology groups of $0\longrightarrow E^{0,1}\longrightarrow Z\times E^{\C}\longrightarrow\mathcal{E}\longrightarrow 0$
we obtain a natural complex linear isomorphism $E^{\C}=H^0(Z,\mathcal{E})$\,. Obviously, $\mathcal{U}=Z\times U$
is a CR submanifold of $\mathcal{E}$ (equivalently, for any
$(J,u)\in\mathcal{U}$, the complex structure of $T_{(J,u)}\mathcal{E}$ induces a linear CR structure
on $T_{(J,u)}\mathcal{U}$\,). Define, similarly, $\mathcal{E}'$ and $\mathcal{U}'$ and note that \eqref{e:cr_q_lin_map} holds,
for any $J\in Z$, if and only if there exists a (necessarily unique) morphism of complex vector
bundles $\mathcal{T}:\mathcal{E}\to\mathcal{E}'$, over $T$, such that $\mathcal{T}|_{\mathcal{U}}=T\times t$.
Thus, if (ii) holds then,
by identifying $Z=Z'(=\C\!P^1)$\,, $T={\rm Id}_Z$, there exists a
holomorphic section $\mathcal{T}$ of ${\rm Hom}_{\C\!}(\mathcal{E},\mathcal{E}')$ such that
$\mathcal{T}|_{\mathcal{U}}={\rm Id}_Z\times t$. Then $\mathcal{T}$ induces a complex linear map
$\t:E^{\C}\to(E')^{\C}$ such that $\t|_U=t$ and, for any $J\in Z$, we have
$\t(E^J)\subseteq(E')^J$; moreover, $\t$ descends to the complex linear map
$\mathcal{T}_J:(E,J)\to(E',J)$\,. It follows that $\t=\mathcal{T}_J\oplus\mathcal{T}_{-J}$\,, $(J\in Z)$\,, and, in particular,
$\t$ is the complexification of some (real) linear map $\widetilde{t}:E\to E'$. This completes
the proof of (ii)$\Longrightarrow$(i)\,.
\end{proof}

\indent
Let $E$ be a quaternionic vector space, $E^{0,1}$ the holomorphic vector subbundle
of $Z\times E^{\C}$ whose fibre over each $J\in Z$ is $E^J$, and $\mathcal{E}$ the
quotient of $Z\times E^{\C}$ through $E^{0,1}$.

\begin{defn}
1) Let $(U,E,\iota)$ be a CR quaternionic vector space. $\mathcal{U}=E^{0,1}\cap(Z\times U^{\C})$ is
called \emph{the holomorphic vector bundle of $(U,E,\iota)$}\,.\\
\indent
2) Let $(U,E,\r)$ be a co-CR quaternionic vector space.\;\emph{The holomorphic vector bundle
of $(U,E,\r)$} is the dual of the holomorphic vector bundle of $(U^*,E^*,\r^*)$\,.
\end{defn}
\indent
Next, we explain how the holomorphic vector bundle $\mathcal{U}$ of a co-CR quaternionic vector space $(U,E,\r)$
can be constructed directly, without passing to the dual CR quaternionic vector space.
For this, note that, with the same notations as above, $\r$ induces an injective morphism of holomorphic vector bundles
$E^{0,1}\to Z\times U^{\C}$. Then $\mathcal{U}$ is the quotient of $Z\times U^{\C}$ through $E^{0,1}$.
Furthermore, we have the following commutative diagram, where $\mathcal{R}$ is the `twistorial representation'
of $\r$:
\begin{displaymath}
\xymatrix{
         &                       &   0  \ar[d]                    &   0  \ar[d]               &  \\
         &                       & Z\times{\rm ker}\r^{\C} \ar[r] \ar[d] & Z\times{\rm ker}\r^{\C} \ar[d] &  \\
0 \ar[r] & E^{0,1} \ar[r] \ar[d] & Z\times E^{\C} \ar[r] \ar[d]^{{\rm Id}_Z\times\r^{\C}}  & \mathcal{E} \ar[r] \ar[d]^{\mathcal{R}} & 0 \\
0 \ar[r] & E^{0,1} \ar[r]        & Z\times U^{\C} \ar[r] \ar[d]   & \mathcal{U} \ar[r] \ar[d] & 0 \\
         &                       &   0                            &   0                       &   }
\end{displaymath}
Then, similarly to the proof of Proposition \ref{prop:cr_q_lin_map}\,, there exists a natural complex
linear isomorphism $U^{\C}=H^0(Z,\mathcal{U})$\,. Moreover, the conjugation of $U^{\C}$ ($E^{\C}$)
induces a conjugation (that is, an involutive antiholomorphic diffeomorphism) on\/ $\mathcal{U}$ ($\mathcal{E}$)
which descends to the antipodal map on $Z$. Then there exists a natural isomorphism between $U$
and the vector space of holomorphic sections of\/ $\mathcal{U}$ which intertwines the conjugations.\\
\indent
Also, note that $H^1(Z,\mathcal{U})=0$ (this follows, for example, from the long exact sequence of cohomology
groups determined by the second row of the above diagram).\\
\indent
The following fact will be used later on.

\begin{prop} \label{prop:cr_q_holo_factor}
Let $(U,E,\iota)$ be a CR quaternionic vector space and let $\mathcal{U}$ be its holomorphic vector bundle.\\
\indent
The quotient of $Z\times U^{\C}$\! through $\mathcal{U}$ is isomorphic to $2k\ol(1)$\,, where $\dim E=4k$\,.
\end{prop}
\begin{proof}
The dual of the quotient of $Z\times U^{\C}$\! through $\mathcal{U}$ is the annihilator $\mathcal{V}$ of\/ $\mathcal{U}$ in $Z\times U^{\C}$\,.
By definition, $\mathcal{U}=\iota^{-1}\bigl(E^{0,1}\bigr)$ and, hence, $\mathcal{V}=\iota^*\bigl((E^*)^{0,1}\bigr)=(E^*)^{0,1}=2k\ol(-1)$\,.
\end{proof}

\indent
Let $(U,E,\r)$ and $(U',E',\r')$ be co-CR quaternionic vector space and let\/ $\mathcal{U}$ and\/ $\mathcal{U}'$,
respectively, be their holomorphic vector bundles. Let $t:(U,E,\r)\to(U',E',\r')$ be a co-CR quaternionic linear map,
with respect to some map $T:Z\to Z'$.
Then $t$ induces a morphism
of holomorphic vector bundles $\mathcal{T}:\mathcal{U}\to\mathcal{U}'$ which intertwines the conjugations.
Furthermore, the above diagram can be easily generalized with $t$ and $\mathcal{T}$ instead of $\r$
and $\mathcal{R}$.

\begin{thm} \label{thm:cr_co-cr_q_vector_spaces}
There exists a covariant functor $\Fa$ from the category of co-CR quaternionic vector spaces,
whose morphisms are the co-CR quaternionic linear maps, to the category of holomorphic vector
bundles over\/ $\C\!P^1$, given by \mbox{$\Fa(U)=\mathcal{U}$} and $\Fa(t)=\mathcal{T}$.\\
\indent
Moreover, if $\Fa(U)=\mathcal{U}$ and $\Fa(U')=\mathcal{U}'$ then for any morphism of holomorphic vector bundles
$\mathcal{T}:\mathcal{U}\to\mathcal{U}'$, which intertwines the conjugations, there exists a unique
co-CR quaternionic linear map $t:U\to U'$ such that $\Fa(t)=\mathcal{T}$.\\
\indent
Therefore two co-CR quaternionic vector spaces are isomorphic if and only if there exists an
isomorphism, which intertwines the conjugations, between their holomorphic vector bundles.
Furthermore, for any positive integer $k$ there exists a nonempty finite set of
isomorphism classes of co-CR quaternionic vector spaces $(U,E,\r)$ with $\dim\!E=4k$.
\end{thm}
\begin{proof}
Let $(U,E,\r)$ and $(U',E',\r')$ be co-CR quaternionic vector spaces and let let $\mathcal{U}$ and $\mathcal{U}'$,
respectively, be their holomorphic vector bundles. Let $\mathcal{T}:\mathcal{U}\to\mathcal{U}'$ be a morphism of
holomorphic vector bundles which intertwines the conjugations. Then $\mathcal{T}$ induces, through the isomorphisms
$U^{\C}=H^0(Z,\mathcal{U})$\,, $U'^{\,\C}=H^0(Z',\mathcal{U}')$\,, a linear map $t:U\to U'$.
Furthermore, $t$ satisfies the dual of assertion (ii) of Proposition \ref{prop:cr_q_lin_map} (with respect to the identity map
of $\C\!P^1$). Thus, $t$ is co-CR quaternionic linear and, obviously, $\Fa(t)=\mathcal{T}$.
The last assertion follows from Proposition \ref{prop:co-cr_q_holo_bundle}\,, below.
\end{proof}

\begin{prop} \label{prop:co-cr_q_holo_bundle}
Let $(U,E,\r)$ be a co-CR quaternionic vector space and let $\mathcal{U}$ be its holomorphic vector bundle,
$\dim\!E=4k$, $\dim({\rm ker}\r)=l$.\\
\indent
Then $\mathcal{U}=\bigoplus_{j=1}^{l+1}a_j\ol(j)$\,, where $a_1,\ldots,a_{l+1}$ are nonnegative integers
satisfying the relations $\sum_{j=1}^{l+1}a_j=2k-l$ and $\sum_{j=1}^{l+1}ja_j=2k$; moreover, $a_j$ is even if\/ $j$ is odd.
Conversely, any holomorphic vector bundle of this form corresponds to a co-CR quaternionic vector space.
\end{prop}
\begin{proof}
As $H^1(Z,\mathcal{U})=0$\,, from the theorem of Birkhoff and Grothendieck (see \cite{Gro-holo_bundles}\,) it follows that
$\mathcal{U}=\bigoplus_{j=-1}^ra_j\ol(j)$ for some nonnegative integers $a_{-1},\,a_0,\ldots,a_r$\,.
Let $J_1,\ldots,J_p\in Z$ be distinct, $(p\geq1)$\,. Under the isomorphism \mbox{$U^{\C}=H^0(Z,\mathcal{U})$}
the complex vector space $\bigcap_{j=1}^p\r\bigl(E^{J_p}\bigr)$ corresponds to the space
of holomorphic sections of\/ $\mathcal{U}$ which are zero at $J_1,\ldots,J_p$\,. Consequently, the complex dimension
of $\bigcap_{j=1}^p\r\bigl(E^{J_p}\bigr)$ is equal to $h_0\bigl(Z,\mathcal{U}\otimes\ol(-p)\bigr)$\,,
where $h_0=\dim_{\C}\!\!H^0$. Hence,
\begin{equation} \label{e:h0}
h_0\bigl(Z,\mathcal{U}\otimes\ol(-1)\bigr)=\sum_{j=1}^rja_j=2k\,\quad , \quad
h_0\bigl(Z,\mathcal{U}\otimes\ol(-2)\bigr)=\sum_{j=2}^r(j-1)a_j=l\,.
\end{equation}
The second relation of \eqref{e:h0} implies that $r-1\leq l$. By combining the first relation of \eqref{e:h0}
with the fact that the degree of\/ $\mathcal{U}$
is $\sum_{j=-1}^rja_j=2k$ we obtain $a_{-1}=0$\,. Now, from \eqref{e:h0} and the fact that the complex rank of\/ $\mathcal{U}$
is $\sum_{j=0}^ra_j=2k-l$ we obtain $a_0=0$\,. The fact that $a_j$ is even if $j$ is odd follows from \cite[(10.7)]{Qui-QJM98}\,.\\
\indent
The last statement is an immediate consequence of Example \ref{exm:co-cr_q_holo_bundle}\,, below.
\end{proof}

\indent
The duals of Theorem \ref{thm:cr_co-cr_q_vector_spaces} and Proposition \ref{prop:co-cr_q_holo_bundle}
can be easily formulated.\\
\indent
The definition of (co-)CR vector subspace can be easily extended to give the
corresponding notion of (co-)CR quaternionic vector subspace.\\
\indent
Now, we give {\it examples of holomorphic vector bundles of (co-)CR quaternionic vector spaces}.

\begin{exm} \label{exm:co-cr_q_holo_bundle}
1) The holomorphic vector bundles of the CR quaternionic vector spaces $(\Hq,\Hq)$ and $({\rm Im}\Hq,\Hq)$ are
$2\ol(-1)$ and $\ol(-2)$\,, respectively.\\
\indent
2) Let $(U,E,\r)$ be a co-CR quaternionic vector space whose holomorphic vector bundle $\mathcal{U}$
has complex rank $1$\,; equivalently, $2k-l=1$\,, where $\dim E=4k$\,. Then\/ $\mathcal{U}=\ol(2k)$\,.
Thus, the holomorphic vector bundle of the CR quaternionic vector space $(U_{k-1},\Hq^{\!k})$
of Example \ref{exm:U_k}(1) is $\ol(-2k)$\,. Moreover, from Theorem \ref{thm:cr_co-cr_q_vector_spaces} and \cite[(10.7)]{Qui-QJM98}\,,
we obtain that $(U,E,\r)$ is isomorphic, as a co-CR quaternionic vector space, to the dual of $(U_{k-1},\Hq^{\!k})$\,.\\
\indent
3) It is straightforward to show that the holomorphic vector bundle of the CR quaternionic vector space $(U'_k,\Hq^{\!2k+1})$
of Example \ref{exm:U_k}(2) is $2\ol(-2k-1)$\,, $(k\geq0)$\,.\\
\indent
4) Let $U_1$ and $U_2$ be co-CR quaternionic vector spaces, and let $Z_1$ and $Z_2$ be the spaces
of admissible linear complex structures of the corresponding quaternionic vector spaces, respectively.
Then any orientation preserving isometry $Z_1=Z_2$ determines a unique linear co-CR quaternionic structure
on $U_1\oplus U_2$ which restricts to the given linear co-CR quaternionic structures on $U_1$ and $U_2$\,.
The resulting co-CR quaternionic vector space is called \emph{the direct sum (product)} of the co-CR quaternionic
vector spaces $U_1$ and $U_2$\,, with respect to the isometry $Z_1=Z_2$\,.
Let\/ $\mathcal{U}_1$ and\/ $\mathcal{U}_2$ be the holomorphic vector bundles of $U_1$ and $U_2$\,, respectively.
Then\/ $\mathcal{U}_1\oplus\mathcal{U}_2$ is the holomorphic vector bundle of $U_1\oplus U_2$\,.
\end{exm}

\indent
In the following classification result, the notations are as in Example \ref{exm:U_k}\,,

\begin{cor} \label{cor:classification_cr_q}
Let $U$ be a CR quaternionic vector space. Then there exists a CR quaternionic linear isomorphism between $U$
and a direct sum in which each term is of the form $U_k$ or\/ $U'_k$ for some natural number $k$\,;
moreover, the terms are unique, up to order.\\
\indent
Furthermore, let $W_n$ be the CR quaternionic vector subspace of\/ $U$ corresponding to the direct sum of all terms up to
$U_{(n/2)-1}$ or\/ $U'_{(n-1)/2}$\,, according to $n$ even or odd, respectively, $(n>1)$\,.
Then the filtration $\{W_n\}_{n>1}$ is canonical (that is, it doesn't depend on the isomorphism).
\end{cor}
\begin{proof}
This is an immediate consequence of Theorem \ref{thm:cr_co-cr_q_vector_spaces}\,, Proposition \ref{prop:co-cr_q_holo_bundle}\,,
Example \ref{exm:co-cr_q_holo_bundle} and \cite[Proposition 2.4]{Gro-holo_bundles}\,.
\end{proof}

\indent
The dual of Corollary \ref{cor:classification_cr_q} can be easily formulated.

\begin{prop} \label{prop:cr_q_are_generic}
Let $k$ and $l$ be positive integers, $l<2k$\,. The set of CR quaternionic vector subspaces $U$ of\/ $\Hq^{\!k}$
of codimension $l$ is a nonempty open set of the Grassmannian of real vector subspaces
of codimension $l$ of $\Hq^{\!k}$.
\end{prop}
\begin{proof}
Let $a,b\in\mathbb{N}$ with $b<2k-l$ and such that $2k=(2k-l-b)a+b(a+1)$\,. Obviously,
if $a$ is even then $b$ is even, whilst if $a$ is odd then $2k-l-b$ is even.\\
\indent
By Proposition \ref{prop:co-cr_q_holo_bundle}\,, there exists a CR quaternionic vector subspace $U$ of $\Hq^{\!k}$
of codimension $l$ whose holomorphic vector bundle is\/ $\mathcal{U}=(2k-l-b)\ol(a)\oplus\,b\,\ol(a+1)$\,.\\
\indent
{}From Theorem \ref{thm:cr_co-cr_q_vector_spaces} we obtain that the closed subgroup $G$ of ${\rm GL}(k,\Hq)$
which preserves $U$ is isomorphic to an open set of the space of real sections of ${\rm End}(\mathcal{U})$\,.
But ${\rm End}(\mathcal{U})=\bigl((2k-l-b)^2+b^2\bigr)\ol\oplus b(2k-l-b)\ol(1)\oplus b(2k-l-b)\ol(-1)$
and, hence, $\dim G=(2k-l)^2$.\\
\indent
Therefore the orbit of $U$ under ${\rm GL}(k,\Hq)$ in the Grassmannian of real vector subspaces
of codimension $l$ of $\Hq^{\!k}$ has dimension $4k^2-(2k-l)^2=l(4k-l)$\,.
\end{proof}

\indent
It can be shown that the set of CR quaternionic vector subspaces $U$ of $\Hq^{\!k}$ of codimension $l$\,,
$(l<2k)$\,, is a nonempty open set of the Grassmannian of real vector subspaces of codimension $l$ of $\Hq^{\!k}$,
with respect to the Zariski topology (induced by the Pl\"ucker embedding). Furthermore, by duality, we obtain the
corresponding fact for co-CR quaternionic vector spaces.

\begin{rem}
The co-CR quaternionic and CR quaternionic vector spaces define augmented and strengthen $\Hq$-modules, in the sense of
\cite{Joy-QJM98} and \cite{Qui-QJM98}\,, respectively. However, the associated holomorphic vector bundles introduced by us are different
from the sheaves introduced in \cite{Qui-QJM98}\,.
\end{rem}

\section{CR quaternionic manifolds and their CR twistor spaces} \label{section:cr_q}

A (smooth) \emph{bundle of associative algebras} is a vector bundle whose typical fibre
is a (finite-dimensional) associative algebra and whose structural group is the group of automorphisms
of the typical fibre. Let $A$ and $B$ be bundles of associative algebras. A morphism of vector bundles $\r:A\to B$ is called
a \emph{morphism of bundles of associative algebras} if $\r$ restricted to each fibre is a morphism of
associative algebras.\\
\indent
Recall that a \emph{quaternionic vector bundle} over a manifold $M$ is a real vector bundle $E$ over $M$ endowed
with a pair $(A,\r)$ where $A$ is a bundle of associative algebras, over $M$, with typical fibre
$\Hq$ and $\r:A\to{\rm End}(E)$ is a morphism of bundles of associative algebras; we say that $(A,\r)$
is a \emph{linear quaternionic structure on $E$} (see~\cite{Bon}\,). Standard arguments (see \cite{IMOP}\,)
apply to show that a quaternionic vector bundle of (real) rank $4k$ is just
a (real) vector bundle endowed with a reduction of its structural group to ${\rm Sp}(1)\cdot{\rm GL}(k,\Hq)$\,.\\
\indent
Note that, a manifold is \emph{almost quaternionic} if and only if its tangent bundle is endowed with a linear quaternionic structure
(see \cite{Sal-dg_qm}\,, \cite{IMOP}\,).\\
\indent
Let $(A,\r)$ be a linear quaternionic structure on the vector bundle $E$. We denote $Q=\r({\rm Im}A)$ and by $Z$ the sphere bundle of $Q$.\\
\indent
Recall, also, that an \emph{almost CR structure} on a manifold $M$ is a complex vector subbundle $\Cal$ of $T^{\C}\!M$
such that $\Cal\cap\overline{\Cal}=\{0\}$\,. An \emph{(integrable almost) CR structure}
is an almost CR structure whose space of sections is closed under the usual bracket.

\begin{defn} \label{defn:almost_cr_q}
Let $E$ be a quaternionic vector bundle on a manifold $M$ and let $\iota:TM\to E$ be an injective
morphism of vector bundles. We say that $(E,\iota)$ is an \emph{almost CR quaternionic structure} on $M$ if $(E_x,\iota_x)$ is a linear
CR quaternionic structure on $T_xM$, for any $x\in M$.\\
\indent
An \emph{almost CR quaternionic manifold} is a manifold endowed with an almost CR quaternionic structure.
\end{defn}

\indent
If $\iota$ is an isomorphism then the definition gives the notion (see \cite{IMOP}\,) of almost quaternionic manifold.

\begin{exm} \label{exm:almost_cr_q}
1) Let $(M,c)$ be a three-dimensional conformal manifold and let $L=\bigl(\Lambda^3TM\bigr)^{1/3}$ be the line bundle of $M$.
Then, $E=L\oplus TM$ is an oriented vector bundle of rank four endowed
with a (linear) conformal structure such that $L=(TM)^{\perp}$. Therefore $E$ is a quaternionic vector bundle
and $(M,E,\iota)$ is an almost CR quaternionic manifold, where $\iota:TM\to E$ is the inclusion.
Moreover, any almost CR quaternionic structure on a three-dimensional manifold $M$ is obtained this way from a
conformal structure on $M$.\\
\indent
2) Let $M$ be a hypersurface in an almost quaternionic manifold $N$. Then $(TN|_M,\iota)$ is an
almost CR quaternionic structure on $M$, where $\iota:TM\to TN|_M$ is the inclusion.\\
\indent
3) More generally, let $N$ be an almost quaternionic manifold. Let $M$ be a submanifold
of $N$ such that, at some point $x\in M$, we have that $(T_xM,E_x,\iota_x)$ is a CR quaternionic vector space,
where $\iota$ is the inclusion $TM\to TN|_M$. Then, Proposition \ref{prop:cr_q_are_generic} implies
that by passing, if necessary, to an open neighbourhood of $x$ in $M$ we have that $(TN|_M,\iota)$
is an almost CR quaternionic structure on $M$.
\end{exm}

\begin{defn} \label{defn:cr_q_connection}
Let $(M,E,\iota)$ be an almost CR quaternionic manifold.  An \emph{almost quaternionic connection} $\nabla$ on $(M,E,\iota)$
is a connection on $E$ which preserves $Q$\,; if, further, $\nabla$ is torsion-free (i.e., $\dif^{\nabla}\!\iota=0$\,)
then it is a \emph{quaternionic connection}.
\end{defn}

\indent
It is well known that if $M$ is an almost quaternionic manifold endowed with an almost quaternionic connection $\nabla$
then the sphere bundle $Z$ of $M$ becomes equipped by an almost complex structure $\mathcal I^\nabla$,
which is uniquely defined and integrable if one uses quaternionic connections;
$(Z,\mathcal I^\nabla)$ is \emph{the twistor space} of $M$, \cite{Sal-dg_qm}.\\
\indent
For any almost CR quaternionic manifold $(M,E,\iota)$\,, endowed with an almost quaternionic connection $\nabla$,
we introduce a natural almost CR structure on the total space of the bundle $\p:Z\to M$ of admissible linear complex structures on $E$.\\
\indent
For any $J\in Z$, let $\mathcal{B}_J\subseteq T^{\C}_J\!Z$ be the horizontal lift, with respect to $\nabla$,
of $\iota^{-1}\bigl(E^J\bigr)$, where $E^J\subseteq E^{\C}_{\p(J)}$ is the eigenspace of $J$ corresponding
to $-{\rm i}$\,. Then, by defining $\Cal_J=\mathcal{B}_J\oplus({\rm ker}\dif\!\p)^{0,1}_J$, $(J\in Z)$\,,
we obtain an almost CR structure $\Cal$ on $Z$.\\
\indent
The following definition is motivated by \cite[Remark 2.10(2)\,]{IMOP}\,.

\begin{defn} \label{defn:cr_q}
An \emph{(integrable almost) CR quaternionic structure} on $M$ is a triple $(E,\iota,\nabla)$\,,
where $(E,\iota)$ is an almost CR quaternionic structure on $M$ and $\nabla$ is an almost quaternionic connection
of $(M,E,\iota)$ such that the almost CR structure $\Cal$ is integrable.\\
\indent
A \emph{CR quaternionic manifold} is a manifold endowed with a CR quaternionic structure.\\
\indent
If $(M,E,\iota,\nabla)$ is a CR quaternionic manifold then $(Z,\Cal)$ is its \emph{twistor space}.
\end{defn}

\indent
{}From \cite[Remark 2.10]{IMOP} it follows that a CR quaternionic structure $(E,\iota,\nabla)$ for which $\iota$ is an isomorphism
is a quaternionic structure.\\
\indent
With the same notations as in Definition \ref{defn:cr_q}\,, the quadruple $(Z,M,\p,\Cal)$ is an almost twistorial structure which
we call \emph{the almost twistorial structure of $(M,E,\iota,\nabla)$}
(see \cite{LouPan-II} for the definition of almost twistorial structures).

\begin{prop} \label{prop:cr_q}
Let $M$ be endowed with an almost CR quaternionic structure $(E,\iota)$ and let $\nabla$ be an
almost quaternionic connection of $(M,E,\iota)$\,; denote $T^J\!M=\iota^{-1}\bigl(E^J\bigr)$\,, for any $J\in Z$\,.\\
\indent
The following assertions are equivalent:\\
\indent
\quad{\rm (i)} The almost twistorial structure of $(M,E,\iota,\nabla)$ is integrable.\\
\indent
\quad{\rm (ii)} $T\bigl(\Lambda^2(T^J\!M)\bigl)\subseteq E^J$ and
$R\bigl(\Lambda^2(T^J\!M)\bigr)\bigl(E^J\bigr)\subseteq E^J$,
for any $J\in Z$, where $T$ and $R$ are the torsion and curvature forms, respectively, of\/ $\nabla$.
\end{prop}
\begin{proof}
This is an immediate consequence of Theorem \ref{thm:int_res}\,, below.
\end{proof}

\begin{thm} \label{thm:cr_q}
Let $M$ be endowed with an almost CR quaternionic structure $(E,\iota)$\,, $\rank E=4k$, $\dim M=4k-l$\,, $(0\leq l\leq 2k-1)$\,.
Let $\nabla$ be a quaternionic connection on $(M,E,\iota)$\,. If $2k-l\neq2$ then the almost twistorial structure of $(M,E,\iota,\nabla)$ is integrable.
\end{thm}
\begin{proof}
If $2k-l=1$ then, as $T^J\!M$ is one-dimensional, for any $J\in Z$, the proof is an immediate
consequence of Proposition \ref{prop:cr_q}\,. Assume $2k-l\geq3$ and note that, as the complexification of the structural group of $E$ is
${\rm SL}(2,\C\!)\cdot{\rm GL}(2k,\C\!)$\,, we have that, locally, $E^{\C}=E'\otimes_{\C\!}E''$
where $E'$ and $E''$ are complex vector bundles of ranks $2$ and $2k$, respectively. Moreover, we have
$\nabla^{\C}=\nabla'\otimes\nabla''$, where $\nabla'$ and $\nabla''$ are connections on $E'$ and $E''$,
respectively. Also, $Z=P(E')$ such that if $J\in Z_x$ corresponds to the line $[u]\in P(E'_x)$
then the eigenspace of $J$ corresponding to $-{\rm i}$ is equal to $\{u\otimes v\,|\,v\in E''_x\,\}$\,,
$(x\in M)$\,. Then assertion (ii) of Proposition \ref{prop:cr_q} holds if and only if
$R'(X,Y)(u)\in[u]$\,, for any $u\in E'_x\setminus\{0\}$ and $X,Y\in T_xM$, such that
$\iota(X), \iota(Y)\in\{u\otimes v\,|\,v\in E''_x\,\}$\,, $(x\in M)$\,, where $R'$
is the curvature form of $\nabla'$. The proof now follows quickly from the Bianchi identity $R\wedge\iota=0$\,,
by using that $\dim\bigl(\iota^{-1}\bigl(\{u\otimes v\,|\,v\in E''_x\,\}\bigr)\bigr)\geq3$, $(x\in M)$\,.
\end{proof}

\begin{exm} \label{exm:cr_q}
1) Let $(M,c,D)$ be a three-dimensional Weyl space; that is, $(M,c)$ is a conformal manifold and $D$ is a torsion free
conformal connection on it. With the same notations as in Example \ref{exm:almost_cr_q}(1)\,, let $D^L$ be the connection
induced by $D$ on $L$\,. Then $\nabla=D^L\oplus D$ is a quaternionic connection on $(M,E,\iota)$\,; in particular,
$(M,E,\iota,\nabla)$ is a CR quaternionic manifold.\\
\indent
2) If in Examples \ref{exm:almost_cr_q}(2) and \ref{exm:almost_cr_q}(3) we assume $N$ quaternionic then we obtain examples
of CR quaternionic manifolds.
\end{exm}

\section{Quaternionic manifolds as heaven spaces} \label{section:q_heaven_spaces}

\indent
We recall the following definition (cf.\ \cite{AndFre}\,, \cite{Ro-LeB_nonrealiz}\,).

\begin{defn} \label{defn:realiz_cr}
Let $(M,\Cal)$ be a CR manifold, $\dim M=2k-l$\,, $\rank\Cal=k-l$\,.\\
\indent
We say that $(M,\Cal)$ is  \emph{realizable} if $M$ is an embedded submanifold, of codimension $l$\,, of a complex manifold $N$
such that $\Cal=T^{\C}\!M\cap(T^{0,1}N)|_M$.\\
\indent
Suppose, further, that $(M,\Cal)$ is endowed with a conjugation $\t$\,; that is, $\t$ is an involutive CR diffeomorphism from
$(M,\Cal)$ onto $(M,\overline{\Cal})$\,. We say that $(M,\Cal,\t)$ is \emph{realizable} if $(M,\Cal)$ is realizable and
$\t$ is the restriction of a conjugation on the corresponding complex manifold.
\end{defn}

\indent
We have the following straight extension of the notion of realizability.

\begin{defn} \label{defn:realiz_cr_q}
Let $(M,E,\iota,\nabla)$ be a CR quaternionic manifold and let $(Z,\Cal)$ be its twistor space. We say that $(M,E,\iota,\nabla)$ is \emph{realizable} if $M$ is an embedded submanifold of a quaternionic manifold $N$
such that $E=TN|_M$, as quaternionic vector bundles, and $\Cal=T^{\C}\!Z\cap(T^{0,1}Z_N)|_M$, where $Z_N$ is the twistor space of $N$.
\end{defn}

\begin{thm} \label{thm:realiz_cr_q}
Let $(M,E,\iota,\nabla)$ be a CR quaternionic manifold and let $(Z,\Cal,\t)$ be its twistor space, endowed with the
conjugation given, on the fibres of $Z$, by the antipodal map. Then the following assertions are equivalent:\\
\indent
{\rm (i)} $(M,E,\iota,\nabla)$ is realizable.\\
\indent
{\rm (ii)} $(Z,\Cal,\t)$ is realizable.
\end{thm}
\begin{proof}
As the implication (i)$\Longrightarrow$(ii) is trivial, it is sufficient to prove (ii)$\Longrightarrow$(i)\,.\\
\indent
If (ii) holds then $(Z,\Cal)$ is an embedded submanifold, of codimension $l$\,, of a complex manifold $Z_1$ such that
$\Cal=T^{\C}\!Z\cap(T^{0,1}Z_1)|_Z$, where $\dim M=4k-l$\,, $\rank E=4k$\,; in particular, $\dim_{\C}Z_1=2k+1$\,.
Also, $\t$ extends to an involutive anti-holomorphic diffeomorphism $\t_1:Z_1\to Z_1$\,.\\
\indent
Note that, the fibres of $\p:Z\to M$ are complex projective
lines embedded in $Z_1$ as complex submanifolds whose normal bundles, by Proposition \ref{prop:cr_q_holo_factor}\,,
are isomorphic to $2k\ol(1)$\,. Then, \cite{Kod} and \cite[Proposition 2.5]{Ro-LeB_nonrealiz} imply
that $M$ is a submanifold of a complex manifold $N_1$\,, $\dim_{\C\!}N_1=4k$\,, which parametrizes a holomorphic family
of complex projective lines which contains $\{\p^{-1}(x)\}_{x\in M}$. Furthermore, if $t_x\subseteq Z_1$ is the complex
projective line corresponding to any $x\in N_1$ then the (holomorphic) tangent space to $N_1$ at $x$ is canonically
isomorphic to the space of sections of the normal bundle of $t_x$ in $Z_1$\,. Moreover, as $H^1(\C\!P^1,{\rm End}(2k\ol(1)))=0$\,,
we may apply \cite[Theorem 7.4]{KodSpe-deformations_I} to obtain that we may assume the normal bundle of $t_x$ in $Z_1$
isomorphic to $2k\ol(1)$\,, for any $x\in N_1$\,.\\
\indent
As $\t_1$ preserves the fibres of $\p$ and maps the complex projective lines parametrized by $N_1$ to complex projective
lines, from \cite{Kod} we obtain that $\t_1$ induces an involutive anti-holomorphic diffeomorphism
$\s$ of $N_1$\,. Moreover, as the fixed point set $N$ of $\s$ is nonempty $(M\subseteq N)$ we, also, have that $N$ is real analytic
and its complexification is $N_1$\,; in particular, $\dim N=4k$\,. Then $N$ is an almost quaternionic manifold whose
bundle of admissible linear complex structures is $Z_1$\,. Furthermore, the complex structure $\J$ of $Z_1$ and the projection
$\p_1:Z_1\to N$ are such that $(\dif\!\p_1)_J:(T_JZ_1,\J_J)\to(T_{\p_1(J)}N,J)$ is complex linear, for any $J\in Z_1$\,.
Hence, locally, there exist sufficiently many admissible complex structures on $N$ to apply \cite[Theorem 2.4]{AleMarPon-99}\,,
thus, obtaining that $N$ is a quaternionic manifold.\\
\indent
To complete the proof we have to show that the map $M\hookrightarrow N$ is a continuous embedding. This is an immediate
consequence of the facts that $Z$ is embedded in $Z_1$ and that the topologies of $M$ and $N$ are equal to the quotient
topologies of $Z$ and $Z_1$ with respect to the (open) projections $\p:Z\to M$ and $\p_1:Z_1\to N$, respectively.
\end{proof}

\indent
In the real-analytic case, the result of Theorem \ref{thm:realiz_cr_q} can be improved as follows.

\begin{cor} \label{cor:realiz_cr_q_1}
Any real-analytic CR quaternionic manifold $(M,E,\iota,\nabla)$ is realizable.
Moreover, the corresponding embedding into a quaternionic manifold is germ unique
(that is, if $\phi:M\to N$ and $\phi':M\to N'$ are embeddings, satisfying Definition \ref{defn:realiz_cr}\,, there exist open neighbourhoods
$U$ and $U'$ of\/ $\phi(M)$ and $\phi'(M)$\,, respectively, and a quaternionic diffeomorphism $\psi:U\to U'$ such that $\psi\circ\phi=\phi'$).
\end{cor}
\begin{proof}
Any real analytic CR manifold is realizable and the corresponding embedding is germ unique \cite{AndFre}\,; moreover, locally, any
real-analytic CR function is the restriction of a unique holomorphic function on the corresponding complex manifold.
It follows quickly that Theorem \ref{thm:realiz_cr_q} can be applied to obtain that $(M,E,\iota,\nabla)$ is realizable.
The unicity is a consequence of the fact that any quaternionic map is
determined by a holomorphic map between the corresponding twistor spaces, preserving the twistor lines \cite{IMOP}\,.
\end{proof}

\indent
If $(M,E,\iota,\nabla)$ is a real-analytic CR quaternionic manifold we call the corresponding quaternionic manifold
the \emph{heaven space} of $(M,E,\iota,\nabla)$\,.

\begin{cor} \label{cor:realiz_cr_q_2}
Let $(M,E,\iota,\nabla)$ be a CR quaternionic manifold and let $(Z,\Cal,\t)$ be its twistor space, endowed with the
conjugation given, on the fibres of $Z$, by the antipodal map. If $(Z,\Cal,\t)$ is realizable then $(M,E,\iota)$ admits
quaternionic connections $\nabla'$ such that the twistor space of $(M,E,\iota,\nabla')$ is equal to $(Z,\Cal)$\,.
\end{cor}
\begin{proof}
This is an immediate consequence of Theorem \ref{thm:realiz_cr_q}\,.
\end{proof}

\appendix

\section{An integrability result}

\indent
Let $E$ be a vector bundle on a manifold $M$ and let $\a:TM\to E$ be a morphism of vector bundles.
Let $\bigl(P,M,{\rm GL}(n,\R)\bigr)$ be the frame bundle of $E$, where $n=\rank E$.
Define an $\R^n$-valued one-form $\theta$ on $P$ by $\theta(X)=u^{-1}\bigl(\a\bigl(\dif\!\p(X)\bigr)\bigr)$\,,
for any $u\in P$ and $X\in T_uP$, where $\p:P\to M$ is the projection.
Then $\theta$ is the tensorial form which corresponds to the $E$-valued one-form $\a$\,.\\
\indent
Let $\nabla$ be a connection on $E$.

\begin{defn} \label{defn:torsion}
The \emph{torsion (with respect to $\a$)} of $\nabla$ is the $E$-valued two-form
$T$ on $M$ defined by $T=\dif^{\nabla}\!\a$\,; if $T=0$ then $\nabla$ is called \emph{torsion-free}.
\end{defn}

Note that, the tensorial two-form which corresponds to the torsion of $\nabla$ is $\dif\!\theta|_{\H}$\,,
where $\H\subseteq TP$ is the principal connection corresponding to $\nabla$ (cf.\ \cite{KoNo}\,).\\
\indent
The next result will be useful later on.

\begin{lem} \label{lem:dif_canonic}
Let $u_0\in P$, $X\in\H_{u_0}$ and let $u$ be any (local) section of $P$ tangent to $X$
(in particular, $u_{\p(u_0)}=u_0$\,). Then for any vector field\/ $Y$ on $P$, we have
$$u_0\bigl(X(\theta(Y))\bigr)=\nabla_{\dif\!\p(X)}\bigl(\a\bigl(\dif\!\p(Y|_{u(M)})\bigr)\bigr)\;.$$
\end{lem}
\begin{proof}
Let $\check{Y}=\dif\!\p(Y|_{u(M)})$\,. Then $\a(\check{Y})$ is a section of $E$ and $\theta(Y|_{u(M)})$
gives the components of $\a(\check{Y})$ with respect to $u$\,. The proof quickly follows
(cf.\ the relation of the Lemma from \cite[vol 1, page 115]{KoNo}\,).
\end{proof}

\indent
Let $F$ be a complex submanifold of the Grassmannian ${\rm Gr}_q(\C^{\!n})$ on which the complexification
$G$ of the structural group of $E$ acts transitively $(q\leq n)$\,, and assume that $\nabla$ is a $G$-connection.\\
\indent
We shall denote by $(P,M,G)$ the bundle of complex $G$-frames on $E$.
Let $Z=P\times_G\!F$ and suppose that the dimension of $\a^{-1}(p)$ does not depend of $p\in Z$.\\
\indent
Let $\Cal_0\subseteq T^{\C}\!Z$ be horizontal (with respect to $\nabla$) and such that
$\dif\!\p_p(\Cal_0)=\a^{-1}(p)$, for any $p\in Z$. Define $\Cal=\Cal_0\oplus({\rm ker}\dif\!\p)^{0,1}$,
where $\p:Z\to M$ is the projection.\\
\indent
Note that, if the dimension of $\a^{-1}(p)\cap\overline{\a^{-1}(p)}$ does not depend of $p\in Z$
then $(Z,M,\p,\Cal)$ is an almost twistorial structure, in the sense of \cite{LouPan-II}\,.

\begin{thm} \label{thm:int_res}
The following assertions are equivalent:\\
\indent
\quad{\rm (i)} $\Cal$ is integrable.\\
\indent
\quad{\rm (ii)} $T(X,Y)\in p$ and $R(X,Y)(p)\subseteq p$\,, for any $p\in Z$ and $X,Y\in\a^{-1}(p)$\,,
where $T$ and $R$ are the torsion and curvature forms, respectively, of\/ $\nabla$,
and $p\in Z$ is considered a vector space.
\end{thm}
\begin{proof}
Fix $V\in F$ and let $H$ be the closed subgroup of $G$ which preserves $V$\,.
Then $Z=P/H$ and let $\psi:P\to Z$ be the projection. If we denote $\tD=(\dif\!\psi)^{-1}\bigl(\Cal\bigr)$
then, as $\psi$ is a surjective submersion, $\Cal$ is integrable if and only if $\tD$ is integrable. Let $\H\subseteq TP$ be the principal connection corresponding to $\nabla^{\C}$ and let
$\tB(V)\subseteq\H^{\C}$ be such that $\tB(V)_u$ is the horizontal lift of
$\a^{-1}\bigl(u(V)\bigr)$\,, for any $u\in P$. Then, we have $\tD=\tB(V)\oplus(P\times\mathfrak{h})\oplus(P\times\overline{\mathfrak{g}})$\,, where
$\mathfrak{g}$ and $\mathfrak{h}$ are the Lie algebras of $G$ and $H$, respectively, and we have
used that ${\rm ker}\dif\!\p=P\times\mathfrak{g}$\,, with $\p:P\to M$ the projection, and
$\mathfrak{g}^{\C}=\mathfrak{g}\oplus\overline{\mathfrak{g}}$ (as complex Lie algebras). Note that, $R_a\bigl(\tB(V)\bigr)=\tB\bigl(a^{-1}(V)\bigr)$ for any $a\in G$, where $R$ denotes the action of $G$ on $P$.
Also, $A\in\mathfrak{h}$ if and only if $A(V)\subseteq V$ whilst if $A\in\overline{\mathfrak{g}}$
then $A(V)=\{0\}$\,. Thus, by using the same notation for elements of $\mathfrak{g}$ and the corresponding
fundamental vector fields, we obtain the following relation (cf.\ \cite[Proposition III.2.3]{KoNo}\,):
for any $A\in\mathfrak{h}\oplus\overline{\mathfrak{g}}$\,,
\begin{equation*}
\bigl[A,\G\bigl(\tB(V)\bigr)\bigr]\subseteq\G\bigl(\tB\bigl(A(V)+V\bigr)\bigr)=\G\bigl(\tB(V)\bigr)\;.
\end{equation*}
\indent
We have thus shown that assertion (i) holds if and only if, for any sections $X$ and $Y$ of $\tB(V)$,
we have that $[X,Y]$ is a section of $\tD$. Let $\Theta$ and $\O$ be the tensorial forms on $P$ corresponding to $T$ and $R$, respectively.
By applying \cite[Corollary II.5.3]{KoNo}\,, Lemma \ref{lem:dif_canonic} and \cite[Proposition 2.6(b)]{LouPan-II}\,,
we obtain that (i) is equivalent to the condition that $\O(X,Y)\in\mathfrak{h}$ and $\Theta(X,Y)\in V$, for any $X,Y\in\tB(V)$\,.
\end{proof}

\begin{rem}
With notations as in the proof of Theorem \ref{thm:cr_q}\,, the bracket of any
section of $P\times\mathfrak{h}$ and the sections of $\tB(V)$ which are basic with respect to $\psi$
is zero.
\end{rem}

\end{document}